\documentclass{article}
\usepackage{amsthm,amssymb,amsmath}
\usepackage{epsfig}

\title{Spectral radius of finite and infinite planar graphs 
       and of graphs of bounded genus}

\author{Zden\v{e}k Dvo\v{r}\'ak\thanks{Supported in part through a postoctoral
   position at Simon Fraser University.}~\thanks{On leave from: 
   Institute of Theoretical Informatics,
   Charles University, Prague, Czech Republic.}\\
  {Department of Mathematics}\\
  {Simon Fraser University}\\
  {Burnaby, B.C. V5A 1S6} \\
  email: {\tt rakdver@kam.mff.cuni.cz}
\and
  Bojan Mohar\thanks{Supported in part by an NSERC Discovery Grant (Canada),
  by the Canada Research Chair program, and by the
  Research Grant P1--0297 of ARRS (Slovenia).}~\thanks{On leave from:
  IMFM \& FMF, Department of Mathematics, University of Ljubljana, Ljubljana,
  Slovenia.}\\
  {Department of Mathematics}\\
  {Simon Fraser University}\\
  {Burnaby, B.C. V5A 1S6} \\
  email: {\tt mohar@sfu.ca}
}

\newtheorem{theorem}{Theorem}[section]
\newtheorem{lemma}[theorem]{Lemma}
\newtheorem{corollary}[theorem]{Corollary}

\newcommand{\DEF}[1]{{\em #1\/}}
\newcommand{\RR}{\ensuremath{\mathbb R}}
\newcommand{\gam}{g} 

\begin{document}

\maketitle

\begin{abstract}
It is well known that the spectral radius of a tree whose maximum degree
is $D$ cannot exceed $2\sqrt{D-1}$. In this paper we derive similar bounds 
for arbitrary planar graphs and for graphs of bounded genus. 
It is proved that a the spectral radius $\rho(G)$ of a planar graph $G$ of 
maximum vertex degree $D\ge 4$ satisfies $\sqrt{D}\le \rho(G)\le \sqrt{8D-16}+7.75$.
This result is best possible up to the additive constant---we construct an
(infinite) planar graph of maximum degree $D$, whose spectral radius is $\sqrt{8D-16}$.
This generalizes and improves several previous results and solves an open problem proposed by Tom Hayes.
Similar bounds are derived for graphs of bounded genus. 
For every $k$, these bounds can be improved by excluding $K_{2,k}$ as a subgraph.
In particular, the upper bound is strengthened for 5-connected graphs.
All our results hold for finite as well as for infinite graphs.

At the end we enhance the graph decomposition method introduced in the 
first part of the paper 
and apply it to tessellations of the hyperbolic plane. 
We derive bounds on the spectral radius that are close to the true value,
and even in the simplest case of regular tessellations of type $\{p,q\}$
we derive an essential improvement over known results, obtaining
exact estimates in the first order term and non-trivial estimates for
the second order asymptotics.
\end{abstract}

\section{Introduction}

Every tree of maximum degree $D$ is a subgraph of the infinite $D$-regular tree.
This observation immediately implies that the spectral radius of every such tree
is at most $2\sqrt{D-1}$. In this paper we derive similar bounds for arbitrary 
planar graphs and for graphs of bounded genus. This generalizes and improves 
several previous results and solves an open problem proposed by Hayes.
Usually higher connectivity of graphs allows more edges in the graph and 
thus gives rise to graphs with larger spectral radius. However, an interesting 
outcome of our proof is that higher connectivity has converse effect in the
case of planar graphs. The extremal examples for the largest 
spectral radius need many 4-separations, and hence 5-connected graphs, in particular,
allow better upper bounds on the spectral radius.

All graphs in this paper are \DEF{simple}, i.e.\ no loops or multiple edges are
allowed. They can be finite or infinite, but we request that they are locally
finite. In fact, we shall always have a (finite) upper bound on the maximum degree.

It is well known that the edges of every planar graph $G$ can be partitioned into three
acyclic subgraphs. By compactness, this extends to all (locally finite) planar graphs
and implies that $\rho(G)\le 6\sqrt{\Delta-1}$, where $\Delta$ is the maximum degree of $G$.
This bound has been improved by Hayes~\cite{Hayes}. He use the following theorem.

\begin{theorem}[Hayes~\cite{Hayes}]\label{thm-hayes}
Any graph $G$ that has an orientation with maximum indegree $k$ 
(hence also any $k$-degenerate
graph) and $\Delta=\Delta(G)\ge 2k$ satisfies $\rho(G)\le 2\sqrt{k(\Delta-k)}$.
\end{theorem}

Since each planar graph $G$ has an orientation with maximum indegree $3$, this
gives $\rho(G)\le \sqrt{12(\Delta-3)}$. At the 1st CanaDAM conference (Banff, Alberta, 2007), Tom Hayes
asked to what extent the constant factor in his upper bound can be improved.
We answer Hayes' question by proving that $\rho(G)\le\sqrt{8\Delta}+O(1)$
(see Theorem \ref{thm-radiusplan}) and by showing that this bound is essentially best possible. Our bound cannot be improved even when $G$ is bipartite and ``tree-like''
(i.e. with lots of 2-separations).
To some surprise, if the connectivity is increased, the upper bound can be 
strengthened further. Actually, it suffices to exclude $K_{2,k}$ subgraph, where $k = o(\Delta)$.
These results also apply for all graphs of bounded genus, cf. Theorem~\ref{thm-radius}.

In the last section we enhance the graph decomposition method used in this paper 
and apply it to tessellations of the hyperbolic plane, whose graph is $p$-regular. 
We derive lower and upper bounds on the spectral radius that are close to 
each other and asymptotically coincide. 
Even in the simplest case of regular tessellations of type $\{p,q\}$,
previously known bounds were not of the right magnitude asymptotically. Our
estimates are exact in the first order term and also give a non-trivial
terms for the second order asymptotics.
See further discussion about known results in the next section.
It is worth pointing out that $p$-regular graphs of planar tessellations are 
$p$-connected (as proved in \cite{Mo06}). 
It turns out that with $q$ tending to infinity, the spectral radius 
tends to the same value as the spectral radius of the $p$-regular tree. 

We use standard terminology and notation. For a graph $G$ and $v\in V(G)$, $e\in E(G)$,
we denote by $G-v$ and $G-e$ the subgraph of $G$ obtained by deleting $v$ and
the subgraph obtained by removing $e$, respectively. If $e=uv$ is not an edge of
$G$, then we denote by $G+e$ the graph obtained from $G$ by adding the edge $e$.
We denote by $\Delta(G)$ and $\delta(G)$ the maximum and the minimum degree of $G$,
respectively. A graph is said to be \DEF{$d$-degenerate}
if every subgraph $H$ of $G$ has
$\delta(H)\le d$. This condition is equivalent to the requirement that $G$ can be
reduced to the empty graph by successively removing vertices whose degree is 
at most~$d$. If $H$ is a subgraph of $G$, we write $H\subseteq G$.

\section{Motivation and overview of known results}

Our motivation for the study of the spectral radius of planar graphs comes from various directions.

\medskip

\noindent
{\bf (1) Harmonic analysis.}
The spectral radius of infinite planar graphs, in particular for tesselations
of the hyperbolic plane, is of great interest in harmonic analysis.
We refer to \cite{MW} and to \cite{Wo1,Wo2} for an overview. 

Tessellations, whose graphs are regular of degree $d$, may have the spectral
radius as large as $d$. However, this happens precisely when the graph is 
\DEF{amenable} (cf., e.g., \cite{Mo91,Wo2}). This is also equivalent to 
the condition that the random walk on the graph is recurrent. This case is well
understood. However, in the case of the tessellations of the hyperbolic plane
(or more general \DEF{Cantor spheres}, see \cite{Mo06}) the random walk is 
transient (Dodziuk \cite{Do}), the isoperimetric number (or the Cheeger constant)
is positive \cite{Mo92}, and the spectral radius is strictly smaller than $d$. 
It can be as small as $2\sqrt{d-1}$ (in the case of the $d$-regular tree). 
Quantitative relationship between these notions is provided via Cheeger inequality
(see, e.g., \cite{BMS-T} or \cite{Wo2}). 
It is thus surprising that the exact values for the
spectral radius of regular tessellations of the hyperbolic plane are not known.
Earlier best results are by \v{Z}uk \cite{Zuk} and Higuchi and Shirai \cite{HS}.
They will be reviewed in the last section, where we present improved bounds.

\medskip

\noindent
{\bf (2) Mixing times of Markov chains.}
Bounds on the spectral radius of planar graphs can be used in the design and analysis of certain Monte Carlo algorithms and have applications not only in the theory
of algorithms but also in theoretical physics. In particular, Hayes~\cite{Hayes} 
and Hayes, Vera, and Vigoda \cite{HVV} used these to prove $O(n\log n)$ mixing time
for the Glauber dynamics for the spin systems on planar graphs. These applications include the Ising model, hard-core lattice gas model, and graph colorings that are important in theoretical physics.

\medskip

\noindent
{\bf (3) An application in geography.}
Boots and Royle \cite{BR} investigated the spectral radius of planar graphs
motivated by an application in geography networks. They conjectured that
for every planar graph, $\rho(G) \le O(\sqrt{n})$, where $n=|G|$, and their
computational experiments suggested that the complete join of $K_2$ and the
path $P_{n-2}$ gives the extremal case. Cao and Vince \cite{CV} made
similar conjecture and proved that $\rho(G) \le 4 + \sqrt{3(n-3)}$.
Yuan \cite{Yuan} and Ellingham and Zha \cite{EZ} found extensions to graphs of 
a fixed genus $g$. It is interesting that all these results are close to best
bounds when there is a vertex whose degree is close to $n$. 
The setting in this paper provides the same type of the results but the bounds
depend on the maximum degree and not the number of vertices.
\medskip

\noindent
{\bf (4) Structural graph theory.}
In the study of graph minors, three basic structures appear when one excludes a fixed
graph $H$ as a minor. The first one is topological---one gets graphs embeddable in
surfaces in which the excluded graph $H$ cannot be embedded.
The second structure are extensions of other structures by adding a bounded number
of new vertices or adding so-called ``vortices'' to the surface structure.
This is somewhat technical and we will not consider it at this point.
The last structure is related to ``tree-like decompositions'' and, in particular,
gives rise to the family of graphs of bounded tree-width. These graphs are
degenerate in the sense that they can be reduced to the empty graph by successively
removing vertices of small degree. One can prove similar bounds on the spectral radius 
as presented in this paper, but the detailed analysis requires additional work and we 
leave details for future work. We refer to \cite{KM} for references concerning
graph minors theory, and to \cite{CdV} for some important relations between
spectral theory and graph minors.

\section{Spectral radius of finite and infinite graphs}

If $V$ is a set, we define $\ell^2(V)$ as the set of all functions 
$f: V\to \RR$ such that $||f||^2=\sum_{v\in V} f(v)^2 < \infty$.
For a graph $G$ with vertex set $V$ and edge set $E$, we define 
the \DEF{adjacency operator} $A=A(G)$ as the linear operator that acts
on $\ell^2(V)$ in the same way as the adjacency matrix by the rule of the
matrix-vector multiplication:
$$
     (Af)(v) = \sum_{\{u,v\}\in E} f(u) \,.
$$
If the degrees of all vertices in $G$ are bounded above by a finite constant $D$,
then this defines a bounded self-adjoint linear operator, whose spectrum is
contained in the interval $[-D,D]$. The supremum of the spectrum is called
the \DEF{spectral radius} of $G$ and is denoted by $\rho(G)$.
We refer to \cite{MW} for more details about the spectrum of infinite graphs,
and refer to \cite{Bi,CDS,GoRo} for results about the spectra of finite graphs.

The following basic result \cite{Mo} enables us to restrict our attention to 
finite graphs if desired.

\begin{theorem}
\label{thm:basic}
If $G$ is an infinite (locally finite) graph, then its spectral radius $\rho(G)$
is the supremum of spectral radii $\rho(H)$ taken over all finite subgraphs
$H$ of $G$, and it is equal to $\sup\{\rho(H_i)\mid i=1,2,\dots\}$, where
$H_1\subseteq H_2\subseteq \cdots$ is any seqence of subgraphs of $G$ such that
$\bigcup_{i\ge1} H_i = G$.
\end{theorem}

The spectral radius is monotone and subadditive. Formally this is stated in the following lemma.

\begin{lemma}
\label{lem:1}
{\rm (a)} If $H\subseteq G$, then $\rho(H)\le \rho(G)$.

{\rm (b)} If $G = K\cup L$, then $\rho(G)\le \rho(K) + \rho(L)$.
\end{lemma}

Application of Lemma \ref{lem:1}(a) to the subgraph of $G$ consisting of
a vertex of degree $\Delta(G)$ together with all its incident edges gives 
a lower bound on the spectral radius in terms of the maximum degree.
Also, the spectral radius is bounded from above by the maximum degree,
so we have the following result:

\begin{lemma}
\label{lem:2}
$\sqrt{\Delta(G)}\le\rho(G)\le \Delta(G)$.
\end{lemma}

\section{Partitioning the edges of an embedded graph}

The {\em weight} $w(e)$ of an edge $e=uv$ is $\deg(u)+\deg(v)$.
We shall use the following results regarding existence of edges of small
weight (also called {\em light edges}) in graphs on surfaces.
If $\Sigma$ is a surface with Euler characteristic of $\chi(\Sigma)$,
then the non-negative integer $\gam=2-\chi(\Sigma)$ is called the {\em Euler genus} of $\Sigma$.

\begin{theorem}[Ivan\v{c}o~\cite{ivanco}]
Let $G$ be a finite graph with minimum degree at least three, embedded in an
orientable surface of Euler genus $\gam$.  Then $G$ contains
an edge $e$ with
$$w(e)\le \begin{cases}
\gam+13 & \text{if $\gam<6$}\\
2\gam+7 & \text{if $\gam\ge 6$}.
\end{cases}$$
\end{theorem}

\begin{theorem}[Jendrol' and Tuh\'arsky~\cite{jetruh}]
Let $G$ be a finite graph with minimum degree at least three, embedded in a
non-orientable surface of Euler genus~$\gam$.  Then $G$ contains
an edge $e$ with
$$ w(e)\le \begin{cases}
2\gam+11 & \text{if\/ $1\le\gam\le 2$}\\
2\gam+9 & \text{if\/ $3\le\gam\le 5$}\\
2\gam+7 & \text{if\/ $\gam\ge 6$.}
\end{cases}$$
\end{theorem}

\noindent
Let us define 
$$d(\gam)=\begin{cases}
10 & \text{if $\gam\le 1$}\\
12 & \text{if $2\le \gam\le 3$}\\
2\gam+6 & \text{if $4\le\gam\le 5$}\\
2\gam+4 & \text{if $\gam\ge 6$}.
\end{cases}
$$
We conclude the following:

\begin{corollary}\label{cor-light}
Let $G$ be a finite graph with minimum degree at least three, embedded in a
surface of Euler genus $\gam$.  Then $G$ contains an edge $uv$ such that
$\deg(u)+\deg(v)\le d(\gam)+3$, and hence both $u$ and $v$ have degree at most $d(\gam)$.
\end{corollary}

We show the following decomposition result for the graphs embedded in a fixed surface:

\begin{theorem}\label{thm-decomp}
Let $G$ be a finite graph embedded in a surface of Euler genus~$\gam$. Let
$s=d(\gam)$ and for each vertex $v\in V(G)$, let
$\hat\delta(v) = \min\{\deg(v),s\}$. Then $G$ can be decomposed as follows:

\begin{itemize}
\item[\rm (a)] $G=T\cup L$, where $T$ is a $2$-degenerate graph and\/
$\hat\delta(v)-2\le\deg_L(v)\le \hat\delta(v)$ for each vertex $v\in V(G)$.
\item[\rm (b)] $G=T\cup L$, where $T$ is a $2$-degenerate graph and\/
$\deg_L(v)\le \hat\delta(v)-2$ for each vertex $v\in V(G)$ with $\deg(v)\ge 2$,
and $\deg_L(v)=0$ if $\deg(v)\le 1$.
\item[\rm (c)] If\/ $G$ does not contain $K_{2,k}$ $(k\ge2)$ as a subgraph, then
$G=T\cup T_1\cup L$, such that\/ $T$ and\/ $T_1$ are forests,
$\Delta(T_1)\le (k-1)(s-1)+2$,
and\/ $\hat\delta(v)-2\le\deg_L(v)\le \hat\delta(v)$ for each vertex $v\in V(G)$.
\end{itemize}
\end{theorem}

\begin{proof}
Let $G$ be a counterexample with the smallest number of edges.
We may assume that $G$ has no isolated vertices. Then $G$ is connected.
Let us call a vertex $v$ {\em small} if $\deg(v)\le s$.  Let $S$ be the set of all small vertices of $G$, and $S_2\subseteq S$ the set of all
vertices of $G$ of degree at most two.
No two vertices in $S\setminus S_2$ are
adjacent, as otherwise we can express $G-e$ as $T\cup L'$ or $T\cup T_1\cup L'$
and set $L=L'+e$, obtaining a decomposition of $G$. In the cases (a) and (c),
the same reduction works for any small vertices, i.e., no two vertices in $S$
are adjacent to each other.

Next, we claim that $\delta(G)\ge 2$.  Otherwise, let $v$ be a vertex of degree 
one, and let $w$ be its neighbor.  As $G$ is the smallest counterexample, 
there exists a decomposition $G-v=T'\cup L$ or $G-v=T'\cup T_1\cup L$.
In the cases (a) and (c), $w\not\in S$, hence $\deg_L(w) \ge s-2$.  
We let $T=T'+vw$ and obtain a contradiction, as $G$ is supposed to be a counterexample.

In the cases (a) and (b), we similarly conclude that $G$ has minimum degree 
at least three (by adding both edges incident with a vertex of degree 2 into $T$).
Since $G$ does not contain two adjacent small vertices, this contradicts 
Corollary~\ref{cor-light}.

It remains to consider the case (c).
Suppose that $G$ contains an edge $uv$ with $\deg(u)\le k(s-1)+1$ and $\deg(v)=2$,
and let $w$ be the neighbor of $v$ distinct from $u$.
By the minimality of $G$, there exists a decomposition $G-v=T'\cup T_1'\cup L$.
We set $T=T'+vw$ and $T_1=T_1'+uv$.
As $G$ does not contain two adjacent small vertices,
$\deg(u)>s$ and $\deg_L(u)\ge s-2$.  It follows that $\deg_{T_1}(u)\le k(s-1)+1-(s-2)=(k-1)(s-1)+2$, hence $\Delta(T_1)\le (k-1)(s-1)+2$.
This is a contradiction, thus each neigbor of a degree-$2$ vertex has degree at least $k(s-1)+2$.

Let $H$ be the simple graph obtained from $G$ by suppressing the degree-$2$ vertices and eliminating the arising parallel edges
(note that the multiplicity of each such edge is at most $k$, as otherwise $G$ would contain $K_{2,k}$ as a subgraph).
If $v\in V(H)$ is not adjacent to a $2$-vertex in $G$ (in particular, if 
$v\in S\setminus S_2$), then $\deg_H(v)=\deg_G(v)\ge 3$.
On the other hand, if $v$ is adjacent to a $2$-vertex, then we conclude that $\deg_H(v)\ge \frac{\deg_G(v)}{k}\ge\frac{k(s-1)+2}{k}\ge 3$.
It follows that the minimum degree of $H$ is at least three, and by Corollary~\ref{cor-light}, $H$ contains an edge $uv$ with $\deg_H(u)+\deg_H(v)\le s+3$.
We may assume that $\deg_H(u)\le \deg_H(v)$, and thus $\deg_G(u)\le k\deg_H(u)\le k\frac{s+3}{2}\le k(s-1)+1$.
We conclude that $u$ is not adjacent to a degree-$2$ vertex in $G$, and hence $\deg_G(u)=\deg_H(u)\le s$ and $u$ is small.
It follows that $uv\in E(G)$ and $v$ is not small, thus $\deg_G(v)>\deg_H(v)$ and $v$ is adjacent to a degree-$2$ vertex in $G$, and $\deg_G(v)\ge k(s-1)+2$.
However, using the fact that $u$ and $v$ do not have a common neighbor of degree $2$, we get
$\deg_H(v)\ge 1 + \lceil(\deg_G(v)-1)/k\rceil\ge 1 + \lceil (k(s-1) + 1)/k\rceil = s+1$, which is a contradiction.
\end{proof}

Consider a decomposition of the graph $K_{3,n}$ into a $2$-degenerate graph $T$ and a graph $L$ of maximum degree $s$.
Let $a_1$, $a_2$ and $a_3$ be the three vertices of degree $n$ and let $B$ be the set of $n$ vertices
of degree three.  Let $B'\subseteq B$ be the set of vertices that are not incident with an edge of $L$.
Since the maximum degree of $L$ is $s$, we obtain $|B'|\ge n-3s$.  As $K_{3,3}$ is not $2$-degenerate, $|B'|\le 2$.
Therefore, $n-3s\le 2$, and $n\le 3s+2$.
As $K_{3,2\gam+2}$ can be embedded in a surface of Euler genus $\gam$ (Ringel \cite{Ri65a}), 
it is not possible to improve the bound on the maximum degree of $L$ in such a decomposition below
$\tfrac{2}{3}\gam$, i.e., $\Delta(L)=\Omega(\gam)$.

\section{Spectral radius of embedded graphs}

We now use the decomposition theorem to obtain a bound on the spectral radius of
graphs of bounded genus. In all proofs we assume that the graph $G$ is finite. 
However, the proof given for the finite case extends to infinite graphs by
applying Theorem \ref{thm:basic} and taking the limit over larger and larger 
finite subgraphs.

\begin{theorem}
\label{thm-radius}
Let $G$ be a graph embedded in a surface of Euler genus~$\gam$.  
\begin{itemize}
\item[\rm (a)] If\/ $\Delta(G)\ge d(\gam)+2$, then 
$\rho(G)\le \sqrt{8(\Delta(G)-d(\gam))} + d(\gam)$.
\item[\rm (b)] If\/ $G$ does not contain $K_{2,k}$ $(k\ge2)$ as a subgraph
and $\Delta(G)\ge d(\gam)$, then
$$\rho(G)\le 2\sqrt{\Delta(G)-d(\gam)+1}+2\sqrt{(k-1)(d(\gam)-1)+1}+d(\gam).$$
\end{itemize}
\end{theorem}

\begin{proof}
Let $G=T\cup L$ be a decomposition as guaranteed by Theorem~\ref{thm-decomp}(a).
Note that every vertex of degree $\ge d(\gam)$ satisfies 
$\deg_T(v) = \deg_G(v) - \deg_L(v) \le \deg_G(v) - d(\gam)+2$ and every
vertex of degree $< d(\gam)$ in $G$ satisfies $\deg_T(v)\le 2$.
Thus $\Delta(T)-2\le \Delta(G)-d(\gam)$. By Theorem~\ref{thm-hayes}, $\rho(T)\le 2\sqrt{2\Delta(T)-4}\le 2\sqrt{2(\Delta(G)-d(\gam))}$.
Furthermore, $\rho(L)\le \Delta(L)\le d(\gam)$.  
The bound on $\rho(G)$ in (a) follows therefrom by the subadditivity of the spectral 
radius (Lemma \ref{lem:1}(b)). Part (b) follows similarly from 
Theorem~\ref{thm-decomp}(c).
\end{proof}

The bound of Theorem \ref{thm-radius}(a) can be improved when $\Delta(G)$ is large
by using the decomposition of Theorem~\ref{thm-decomp}(b) instead of (a).
We present this improvement only in the special case of
planar graphs, where another slight improvement is possible.

\begin{theorem}\label{thm-radiusplan}
A planar graph\/ $G$ of maximum degree $\Delta=\Delta(G)\ge 10$ has
$$
    \rho(G) \le \sqrt{8\Delta-80}+2\sqrt{21} < \sqrt{8\Delta-80}+9.17
$$
and
$$
    \rho(G) \le \sqrt{8\Delta-16}+2\sqrt{15} < \sqrt{8\Delta-16}+7.75
$$
Furthermore, if\/ $G$ does not contain a separating $4$-cycle, then
$$
 \rho(G)\le 2\sqrt{\Delta-9}+2\sqrt{19}+2\sqrt{21} < 2\sqrt{\Delta-9}+17.883.
$$
\end{theorem}

\begin{proof}
We proceed as in the proof of Theorem~\ref{thm-radius}, considering 
the decomposition $G=T\cup L$. We may assume that $\Delta \ge 12$ since the
bounds follow easily for $\Delta\le 11$ by using Theorem \ref{thm-hayes}.
We estimate the contribution of $T$ in the same way.  
However, we use Theorem~\ref{thm-hayes}
to bound the spectral radius of $L$.  Every planar graph has an orientation with maximum indegree $3$,
hence $\rho(L)\le 2\sqrt{3(\Delta(L)-3)}\le 2\sqrt{21}< 9.17$.
For the second inequality we apply Theorem~\ref{thm-decomp}(b) instead of (a).

Consider now the case that $G$ does not contain separating $4$-cycles.
If $G$ does not contain $K_{2,3}$ as a subgraph, then the bound follows as in
Theorem~\ref{thm-radius}, using the fact that $\rho(L)\le 2\sqrt{21}$.
So suppose that $K_{2,3}\subseteq G$. As $G$ has no separating $4$-cycles, it is easy to see that $V(G)=V(K_{2,3})$, and thus $\Delta(G)\le 4$. 
\end{proof}

In the estimates of Theorem \ref{thm-radius}, 
we can improve the dependency on the genus using the following lemma:

\begin{lemma}\label{lemma-smalldeg}
Let $\varepsilon>0$ be a real number and let 
$G$ be a finite graph embedded in a surface of Euler genus $\gam$, 
with $\Delta(G)=O(\gam)$. 
Then $\rho(G)=O(\varepsilon^{-1}\gam^{\frac{1+\varepsilon}{2}})$.
\end{lemma}

\begin{proof}
Let $k=\left\lceil \varepsilon^{-1}\right\rceil$.
We construct a decomposition $G=G_1\cup \cdots \cup G_k$, such that 
for $i=1,\ldots, k$, $\Delta(G_i)=O(\gam^{1-\varepsilon (i-1)})$ 
and $G_i$ is $O(\gam^{\varepsilon i})$-degenerate.
By Theorem~\ref{thm-hayes}, $\rho(G_i)=O(\gam^{(1+\varepsilon)/2})$,
and by the subadditivity of the spectral radius, $\rho(G)=O(\varepsilon^{-1}\gam^{(1+\varepsilon)/2})$.

Suppose that we have already constructed graphs $G_1$, $G_2$, \ldots, $G_i$.
If $i=0$, then let $H_0=G$, otherwise let $H_i$ be the complement of
$G_1\cup\cdots\cup G_i$ in $G$, i.e., the subgraph of $G$ consisting of
the edges that do not belong to $G_1\cup\cdots\cup G_i$ and the vertices
incident with these edges.  Let $S_{i+1}$ be the set of vertices 
obtained in the following way: we take a vertex of degree less than
$\gam^{\varepsilon(i+1)}+6$ in $H_i$, add it to $S_{i+1}$, and remove it from $H_i$. 
We repeat this process as long as the graph contains vertices of degree
less than $\gam^{\varepsilon(i+1)}+6$.  We let $G_{i+1}$ consist of the edges incident with at least one vertex of $S_{i+1}$.
This ensures that $G_{i+1}$ is $(\gam^{\varepsilon(i+1)}+6)$-degenerate.
Note that $H_{i+1}=H_i-S_{i+1}$. 

The construction also ensures that for $i\ge 1$, the minimum degree
of $H_i$ is at least $\gam^{\varepsilon i}+6$ (if $H_i\neq\emptyset$, 
which we may assume), hence $2|E(H_i)|\ge (6+\gam^{\varepsilon i})|V(H_i)|$.
On the other hand, as $H_i$ is embedded in a surface of Euler genus $\gam$, $2|E(H_i)|\le 6|V(H_i)| - 12 + 6\gam$,
hence $\Delta(H_i)\le |V(H_i)| = O(\gam^{1-\varepsilon i})$.  Since $G_{i+1}\subseteq H_i$, this implies the claimed upper
bound on the maximum degree of $G_{i+1}$ and completes the proof.
\end{proof}

The exponent in the bound of Lemma~\ref{lemma-smalldeg} cannot be improved below 
$1/2$, as the complete graph on $\Omega(\sqrt{\gam})$ vertices
can be embedded in a surface of Euler genus $\gam$.
Together with the decompositions given by Theorem~\ref{thm-decomp}, 
Lemma~\ref{lemma-smalldeg} gives:

\begin{theorem}\label{thm-radiussqg}
If a graph $G$ has Euler genus $\gam$, then
$$\rho(G)\le \sqrt{8\Delta(G)}+O(\gam^{\frac{1}{2}}\log \gam).$$
If $k$ is a positive integer and $G$ does not contain $K_{2,k}$ as a subgraph, then
$$\rho(G)\le 2\sqrt{\Delta(G)}+O(\gam^{\frac{1}{2}}\log\gam).$$
\end{theorem}

\begin{proof}
We apply the decompositions given by Theorem~\ref{thm-decomp}. For the small
degree subgraph $L$ we use Lemma \ref{lemma-smalldeg} with
$\varepsilon = (\log\gam)^{-1}$ to conclude
that $\rho(L)=O(\varepsilon^{-1}\gam^{\frac{1+\varepsilon}{2}})
=O(\gam^{\frac{1}{2}}\log\gam)$. 
\end{proof}

\section{Lower bounds}

In this section we show that the bounds given by Theorem~\ref{thm-radius} are tight up
to the additive term.  As the spectral radius of an infinite $d$-regular tree
is $2\sqrt{d-1}$, for any $\varepsilon>0$ there exists a finite tree $T$ with
$\rho(T)>2\sqrt{\Delta(T)-1}-\varepsilon$, matching the upper bound
$2\sqrt{\Delta(G)}+O(1)$ for planar graphs excluding $K_{2,k}$.

Let $k$ and $d$, $k<d$, be integers such that $d$ is divisible by $k$.  Consider now the following
sequence of graphs $H^{k,d}_i$.  Let $H^{k,d}_0=K_{k, d-k}$.
The graph $H^{k,d}_i$ contains $k\left(\frac{d-k}{k}\right)^{i+1}$ vertices of degree $k$,
let $S_{i+1}$ be the set of these vertices.  The graph $H^{k,d}_{i+1}$ is obtained from $H^{k,d}_i$
by partitioning $S_{i+1}$ into $k$-tuples in some canonical way, then for each such $k$-tuple $C$ adding
$d-k$ new vertices adjacent to each vertex of $C$ (the newly added vertices form the set $S_{i+2}$).
The infinite graph $H^{k,d}$ is the limit of the sequence of the graphs $H^{k,d}_i$.
See Figure~\ref{fig-hd} for an example with $k=2$ and $d=8$.
The following properties are easy to prove:

\begin{figure}
\center{\epsfbox{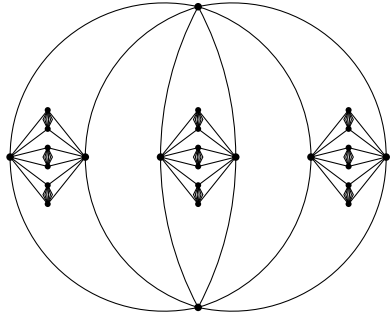}}
\caption{The graph $H^{2,8}$}
\label{fig-hd}
\end{figure}

\begin{itemize}
\item $\Delta(H^{k,d})=d$.
\item $\rho(H^{k,d})=\sup_i \rho(H^{k,d}_i)$.
\item The graphs $H^{2,d}_i$ are planar (assuming the natural partitionings of the sets $S_i$).
\item The graphs $H^{k,d}_i$ are $k$-degenerate.
\end{itemize}

\begin{lemma}\label{lemma-lbound}
The spectral radius of $H^{k,d}$ is $2\sqrt{k(d-k)}$.
\end{lemma}

\begin{proof}
Due to Theorem~\ref{thm-hayes} and the second observation in the previous paragraph,
it suffices to show that $\rho(H^{k,d})\ge 2\sqrt{k(d-k)}$.
Let $A$ be the adjacency operator associated with $H=H^{k,d}$.
In addition to the sets $S_i$ defined during the construction of $H$, let $S_0$ be the
set of $k$ vertices of $H$ of degree $d-k$.  Let us recall that $|S_i|=k\left(\frac{d-k}{k}\right)^i$.
Furthermore, all the edges of $H$ are between the vertices of $S_i$ and $S_{i+1}$, 
for $i=0,1,2,\dots$.
Observe that there are exactly $k|S_{i+1}|=(d-k)|S_i|$ edges between $S_i$ and $S_{i+1}$.

Let $f$ be the function defined by $f(v)=q^i$ for any $v\in S_i$, where $0<q<\sqrt{\frac{k}{d-k}}$.
Note that 
$$
  ||f||^2=\sum_{v\in V(H)}f^2(v)
         =\sum_{i=0}^{\infty} |S_i|q^{2i}
         =\sum_{i=0}^{\infty} k\left(q^2\, \frac{d-k}{k}\right)^i<\infty.
$$
Also, 
\begin{eqnarray*}
  \langle f|Af\rangle &=& 2\sum_{uv\in E(H)} f(u)f(v)\\
      &=& 2\sum_{i=0}^{\infty} (d-k)|S_i|q^{2i+1}\\
      &=& 2q(d-k)\sum_{i=0}^{\infty}|S_i|q^{2i}.
\end{eqnarray*}
It follows from the above calculations that 
$\frac{\langle f|Af\rangle}{||f||^2}=2q(d-k)$ can be arbitrarily close to 
$2(d-k)\sqrt{\frac{k}{d-k}}=2\sqrt{k(d-k)}$.
Therefore, $\rho(H^{k,d})\ge 2\sqrt{k(d-k)}$.
\end{proof}

We conclude that

\begin{itemize}
\item the upper bound in Theorem~\ref{thm-hayes} is best possible for graphs that have an orientation with maximum indegree $k$
(i.e., the graphs with maximum average density at most $k$) and for $k$-degenerate graphs, and
\item as the graph $H^{2,d}$ is planar, the bound $\sqrt{8\Delta}+O(1)$ for the spectral radius of a planar graph as given in Theorem \ref{thm-radius}
is best possible up to the additive term.
\end{itemize}

\section{Hyperbolic tessellations}

In this section, we show how to apply a refined decomposition technique to bound the spectral radius of a special kind of infinite planar graphs.  
For two integers $p,q\ge 3$, where $\tfrac{1}{p}+\tfrac{1}{q} \le \tfrac{1}{2}$,
we call a connected infinite simple plane graph $G$ 
a {\em $(p,\ge q)$-tessellation} if it is $p$-regular and each of its faces has size at least $q$, and
every compact subset of the plane contains only a finite number of its vertices.
If all faces have finite size, then this condition implies that $G$ is 
a one-ended graph, but in the presence of faces of infinite length, $G$ may have
more than one end. We will assume that $p\ge4$ and $q\ge4$.
The cases $p=3$ and $q=3$ could be dealt with (assuming that $\frac{1}{p}+\frac{1}{q}\le \frac{1}{2}$), but the decomposition results
would require slight modificiations.

\begin{lemma}\label{lemma-etes}
Let\/ $C$ be a cycle in a $(p,\ge q)$-tessellation $G$, and let\/ $H$ be 
the subgraph of\/ $G$ contained in the closed disk bounded by $C$.
Let\/ $d=\sum_{u\in V(C)} \deg_H(u)$ and $k=|V(C)|$. 
Then $d < 2(k-1)\frac{q-1}{q-2}$.
\end{lemma}

\begin{proof}
By the definition of $(p,\ge q)$-tessellations, $H$ is finite.
Let $n=|V(H)|$, $e=|E(H)|$ and let $s$ be the number of faces of $H$.  By Euler's formula, $n+s=e+2$.  Furthermore, observe
that $2e=pn-pk+d$ and $2e\ge qs-q+k$.  Combining them, we obtain the following inequality:
$$\frac{d}{p}+2e\left(\frac{1}{2}-\frac{1}{p}-\frac{1}{q}\right)\le k-\frac{k}{q}-1.$$
As $\frac{1}{2}-\frac{1}{p}-\frac{1}{q}\ge 0$ and $2e\ge d$, we get
$$d\,\frac{q-2}{2q}\le k-\frac{k}{q}-1,$$
and hence
$$d\le 2k\frac{q-1}{q-2}-\frac{2q}{q-2} < (2k-2)\frac{q-1}{q-2}.$$
\end{proof}

Note that Lemma~\ref{lemma-etes} implies that $G$ is triangle-free if $q\ge4$.
Let $v$ be a vertex of a $(p,\ge q)$-tessellation $G$.  We define a partition $V_0\cup V_1\cup V_2\cup \cdots$ of vertices of $G$
and a partition $F_0\cup F_1\cup F_2\cup \cdots$ of faces of $G$ in the following way:  let $V_0=\{v\}$ and let $F_0$ consist of
faces incident with $v$.  For each $i>0$, let $V_i$ consist of the vertices incident with the faces in $F_{i-1}$, excluding
those in $V_{i-1}$, and let $F_i$ consist of all faces incident with the vertices of $V_i$, excluding those in $F_{i-1}$.
Let $G_i$ be the subgraph of $G$ induced by $V_i$.  We call the graphs $G_1$, $G_2$, \ldots the \DEF{layers} of $G$ with respect to $v$.

\begin{lemma}\label{lemma-part}
For every $(p,\ge q)$-tessellation $G$ with $p\ge4$ and $q\ge4$ and a vertex 
$v\in V(G)$, the partition $V_0\cup V_1\cup V_2\cup \cdots$ has the following
properties, for each $i>0:$

\begin{itemize}
\item[\rm (a)] The subgraph $G_i$ is either a union of infinite paths, or a cycle.  The face of $G_i$ that contains $v$
is equal to $F_0\cup F_1\cup\cdots\cup F_{i-1}$; the boundary of every other face of $G_i$ is bounded by a connected
component of $G_i$.
\item[\rm (b)] Each vertex of\/ $V_i$ has at most one neighbor in $V_{i-1}$.
\item[\rm (c)] A face belonging to $F_{i-1}$ is incident with at most two vertices in\/ $V_{i-1}$, and if it is incident with two such vertices, 
then they are adjacent in $G_{i-1}$.
\end{itemize}
\end{lemma}

\begin{proof}
For a contradiction, assume that $i$ is the smallest positive integer such that one of the conditions (a), (b) or (c) is violated.
Let us first consider the possibility that condition (b) is false, and let $u\in V_i$ be a vertex with at least two neighbors $w_1,w_2\in V_{i-1}$.
Obviously, $i\ge 2$, and thus $G_{i-1}$ satisfies condition (a).  It follows that $w_1$ and $w_2$ belong to the same component of
$G_{i-1}$.  Let $C$ be the unique cycle in $G_{i-1}+uw_1+uw_2$ such that the disk bounded by $C$ does not contain $v$, and let $H$
be the subgraph of $G$ drawn in the closed disk bounded by $C$.  By the conditions (a) and (b) applied for $i-1$, we conclude that $\deg_H(w)\ge p-1\ge 3$
for each vertex $w\in V(C)$, except for $u$, $w_1$ and $w_2$. 
Let $k=|V(C)|$ and $d=\sum_{w\in V(C)} \deg_H(w)$. By the above,
$$
   d \ge 3(k-3) + \deg_H(u) + \deg_H(w_1) + \deg_H(w_2) \ge 3(k-1).
$$
However, since $q\ge4$, Lemma~\ref{lemma-etes} implies that
$$
   d < 2(k-1)\frac{q-1}{q-2} \le 3(k-1),
$$
a contradiction.

Now, consider the possibility that condition (b) holds, but condition (c) fails. 
As $G$ is triangle-free, 
we conclude that there exists a face $f\in F_{i-1}$ incident with two non-adjacent vertices $w_1$ and $w_2$ in $V_{i-1}$.
Note that $i\ge 2$.  We consider a cycle $C$ contained in the union of $G_{i-1}$ and the boundary of $f$, such that
the disk bounded by $C$ contains neither $v$ nor $f$, and let $H$ be the subgraph of $G$ contained in the closed disk bounded by $C$.
Note that $\deg_H(w)=p$ for any vertex $w\in V(C)\setminus V_{i-1}$, and $\deg_H(w)\ge p-1$ for $w\in V(C)\setminus V(f)$,
i.e., all but at most two vertices $w\in V(C)$ satisfy $\deg_H(w)\ge 3$.  This again contradicts Lemma~\ref{lemma-etes}.

Finally, suppose that (b) and (c) hold.  Consider a vertex $u\in V_i$. 
Similarly as in the case (b), we conclude that $u$ is incident with at most
two faces in $F_{i-1}$ and that if it is incident with two such faces,
then they share an edge $uw$ with $w\in V_{i-1}$.  Also, any edge of $G_i$ is incident with a face in $F_{i-1}$.
It follows that $G_i$ is $2$-regular, and thus it is a union of cycles and infinite paths.  By Lemma~\ref{lemma-etes}
and the property (b) of $G_i$, each disk bounded by a cycle in $G_i$ contains $v$. 
The claim (a) follows, as
$F_0\cup F_1\cup\cdots\cup F_{i-1}$ is a connected subset of the plane.
\end{proof}

We also need the following fractional version of Lemma~\ref{lem:1}:

\begin{lemma}
\label{lemma-fracadd}
Let\/ $G_1$, $G_2$, \ldots, $G_m$ be subgraphs of a graph $G$ such that each
edge of\/ $G$ appears in at least $p$ of the subgraphs. Then $\rho(G)\le \frac{1}{p}\sum_{i=1}^m\rho(G_i)$.
\end{lemma}

\begin{proof}
By the monotonicity, we may assume that each edge of $G$ appears in exactly $p$ of the subgraphs.
Let $A$ be the adjacency operator of $G$ and $A_i$ the adjacency operator of $G_i$ for $1\le i\le m$,
and observe that $A=\frac{1}{p}\sum_{i=1}^mA_i$.

Let $\varepsilon>0$.  There exists a function $f$ such that $||f||=1$ and $\langle f|Af\rangle\ge \rho(G)-\varepsilon$.
By linearity, $\langle f|Af\rangle=\frac{1}{p}\sum_{i=1}^m\langle f|A_if\rangle\le \frac{1}{p}\sum_{i=1}^m\rho(G_i)$.
Therefore,
$\rho(G)-\varepsilon\le \frac{1}{p}\sum_{i=1}^m\rho(G_i)$.  Since this inequality holds for any $\varepsilon>0$,
the claim of the lemma follows.
\end{proof}

We are now ready to estimate the spectral radius of tessellations:

\begin{theorem}\label{thm-tessel}
If\/ $G$ is a $(p,\ge q)$-tessellation with $p\ge4$ and $q\ge4$, then 
$$\rho(G)\le 2\sqrt{p-1}+\frac{2}{q-3}\,.$$
\end{theorem}

\begin{proof}
Choose a vertex $v\in V(G)$ arbitrarily, and consider the layers 
$G_1$, $G_2$, \ldots with respect to $v$.
Let us color a vertex $u\in V_i$ black if $u$ has a neighbor in $V_{i-1}$, and white otherwise.
Let an {\em earthworm} be a maximal subgraph $H$ of $G_1\cup G_2\cup\cdots$ such that
every two vertices of $H$ are joined by a path whose inner vertices are white.
By Lemma~\ref{lemma-part}(a) and (c), all earthworms are paths of length at least $q-3$.
Let $M_1$, $M_2$, \ldots, $M_{q-3}$ be edge-disjoint matchings such that each of them
intersects every earthworm in exactly one edge. 
For $1\le i\le q-3$, consider the graph $T_i=G-M_i$. We claim that 
$T_i$ is a forest. Suppose for a contradiction that $T_i$ contains a cycle $C$.
Let $j$ be the greatest index such that $V(C)\cap V_j\neq\emptyset$.  As $M_i$ contains at least one edge
from each component of $G_j$, $C\not\subseteq G_j$.  Let $P$ be a maximal subpath of $C\cap G_j$.
Since each vertex of $G_j$ has at most one neighbor in $V_{j-1}$, $P$ is not a single vertex.
We conclude that $P$ joins two black vertices of $G_j$ and thus it is a supergraph of at least one earthworm.
Therefore, $M_i\cap P\neq \emptyset$, which is a contradiction.
This proves our claim.

Let $T_{q-2}=G_1\cup G_2\cup\cdots$. Observe that each edge of $G$ belongs to at least
$q-3$ of the graphs $T_1$, $T_2$, \ldots, $T_{q-2}$, $\rho(T_i)\le 2\sqrt{p-1}$ for $1\le i\le q-3$ and $\rho(T_{q-2})=2$.
By Lemma~\ref{lemma-fracadd}, we get $\rho(G)\le 2\sqrt{p-1}+\frac{2}{q-3}$.
\end{proof}

As $q$ goes to infinity, the bound of Theorem~\ref{thm-tessel} aproaches 
$2\sqrt{p-1}$, which is the spectral radius of the $p$-regular infinite tree.
This considerably improves known upper bounds, including the previously best
bound of Higuchi and Shirai \cite{HS}, who proved that
$$
  \rho(G)\le 2\sqrt{(p-2)(1+\tfrac{1}{q-2})}.
$$ 

A non-trivial lower bound on the spectral radius of $p$-regular graphs has
been obtained only for vertex-transitive graphs. Paschke~\cite{Pa}
showed that a vertex transitive $p$-regular graph containing a $q$-cycle has 
spectral radius at least
$$
   \min_{s>0} \,(p-2)\, \phi\biggl(\frac{1+\cosh sq}{\sinh sq\sinh s}\biggr) 
              + 2\cosh s,
$$
where $\phi(t) = \frac{\sqrt{1+t^2} - 1}{t}$. This gives a lower bound of the form
$$
   2\sqrt{p-1} + \frac{2(p-2)}{(p-1)^{(q+1)/2}}\, h(p,q),
$$
where $h$ is a function such that such that $\lim_{p\to\infty} h(p,q)=1$ and $\lim_{q\to\infty} h(p,q)=1$.
The asymptotics (when $p$ or $q$ is large) of this lower bound is different
from our upper bound in Theorem~\ref{thm-tessel} in the ``second order term'' when
$(p,\ge q)$-tessellations are considered. It would be of interest to determine 
the exact behavior.


\begin{thebibliography}{10}

\bibitem{Bi} N.~L.~Biggs, 
Algebraic Graph Theory (2nd ed.),
Cambridge Univ. Press, 1993.

\bibitem{BMS-T} N.~L.~Biggs, B.~Mohar, J.~Shawe-Taylor,
The spectral radius of infinite graphs,
Bull.\ London Math.\ Soc.~20 (1988) 116--120.

\bibitem{BR} B.N. Boots, G.F. Royle, 
A conjecture on the maximum value of the principal eigenvalue of a planar graph,
Geographical Analysis 23 (1991) 276--282.

\bibitem{CV} D. Cao, A. Vince,
The spectral radius of a planar graph,
Linear Algebra Appl. 187 (1993) 251--257. 

\bibitem{CdV} Y. Colin de Verdi\`ere,
Spectres de Graphes, Cours Sp\'ecialis\'es 4, Soc.\ Math.\ France, 1998.

\bibitem{CDS} D.~M.~Cvetkovi\'{c}, M.~Doob and H.~Sachs,
Spectra of Graphs (3rd ed.), Johann Ambrosius Barth Verlag, 1995.

\bibitem{Do} J. Dodziuk, 
Difference equations, isoperimetric inequalities and transience of certain random walks,
Trans. Amer. Math. Soc.~284 (1984), 787--794.

\bibitem{EZ} M. Ellingham, X. Zha,
J. Combin. Theory, Ser. B 78 (2000) 45--56.

\bibitem{GoRo} C. Godsil, G. Royle, 
Algebraic Graph Theory, Springer, 2001.

\bibitem{Hayes} T.~P. Hayes, 
A simple condition implying rapid mixing of single-site dynamics on spin systems,
in ``46th Ann. IEEE Symp. Found. Comp. Sci. (FOCS'06)," IEEE, 2005, pp.~511--520.

\bibitem{HVV} T.~P. Hayes, J.~C. Vera, E. Vigoda,
Randomly coloring planar graphs with fewer colors than the maximum degree,
{\tt arXiv:0706.1530v1}, 2007.

\bibitem{HS} Y. Higuchi, T. Shirai, 
Isoperimetric constants of $(d,f)$-regular planar graphs,
Interdisc.\ Inform.\ Sci.~9 (2003), 221--228.

\bibitem{HoJo} R. A. Horn, C. R. Johnson,
Matrix Analysis, Cambridge Univ.\ Press, Cambridge, 1985.

\bibitem{ivanco} J. Ivan\v{c}o, 
The weight of a graph, 
Ann. Discrete Math. {\bf 51} (1992), 113--116.

\bibitem{jetruh} S. Jendrol' and M. Tuh\'arsky, 
A Kotzig type theorem for non-orientable surfaces, 
Math. Slovaca {\bf 56} (2006), 245--253.

\bibitem{KM} K. Kawarabayashi, B. Mohar,
Some recent progress and applications in graph minor theory,
Graphs Combin. 23 (2007) 1--46.

\bibitem{Mo} B.~Mohar,
The spectrum of an infinite graph,
Linear Algebra Appl.~48 (1982) 245--256.

\bibitem{Mo91} B.~Mohar,
Some relations between analytic and geometric
properties of infinite graphs,
Discrete Math.~95 (1991) 193--219.

\bibitem{Mo92} B.~Mohar,
Isoperimetric numbers and spectral radius of some infinite planar graphs,
Math.\ Slovaca 42 (1992) 411--425.

\bibitem{Mo06} B. Mohar,
Tree amalgamation of graphs and tessellations of the Cantor sphere,
J. Combin.\ Theory Ser.\ B 96 (2006) 740--753.

\bibitem{MW} B.~Mohar, W.~Woess,
A survey on spectra of infinite graphs,
Bull.~London Math.~Soc.~21 (1989) 209--234.

\bibitem{Pa} W.~Paschke,
Lower bound for the norm of a vertex-transitive graph,
Math. Z. 213 (1993) 225--239

\bibitem{Ri65a} G.~Ringel, 
Das Geschlecht des vollst\"andigen paaren Graphen, 
Abh. Math. Sem. Univ. Hamburg 28 (1965) 139--150.

\bibitem{Wo1} W. Woess, 
Random walks on infinite graphs and groups---a survey on selected topics,
Bull. London Math. Soc. 26 (1994) 1--60.

\bibitem{Wo2} W. Woess, 
Random Walks on Infinite Graphs and Groups,
Cambridge Univ. Press, 2000.

\bibitem{Yuan} H. Yuan,
On the spectral radius and the genus of graphs,
J. Combin. Theory, Ser. B 65 (1995) 262--268.

\bibitem{Zuk} A. Zuk,
On the norms of the random walks on planar graphs,
Ann. Inst. Fourier 47, No.5 (1997), 1463--1490.


\end{thebibliography}
\end{document}